\documentclass{amsproc}
\usepackage{amssymb}
\usepackage{amsxtra}
\usepackage[mathscr]{eucal}
\usepackage{graphics}

\numberwithin{equation}{section}

\newtheorem{theorem}{Theorem}[section]
\newtheorem{proposition}[theorem]{Proposition}
\newtheorem{lemma}[theorem]{Lemma}

\theoremstyle{definition}

\newtheorem{definition}[theorem]{Definition}
\newtheorem{example}[theorem]{Example}

\newtheorem{remark}[theorem]{Remark}

\newcommand{\al}{\alpha}
\newcommand{\ardma}{\al_{\rd/\fa}}
\newcommand{\ardmf}{\al_{\rd/\<f\>}}
\newcommand{\aut}{\operatorname{aut}}
\newcommand{\ba}{\mathbf{a}}
\newcommand{\bm}{\mathbf{m}}
\newcommand{\bn}{\mathbf{n}}
\newcommand{\bo}{\boldsymbol{\omega}}
\newcommand{\bs}{\mathbf{s}}
\newcommand{\bt}{\mathbf{t}}
\newcommand{\bu}{\mathbf{u}}
\newcommand{\bxi}{\boldsymbol{\xi}}
\newcommand{\CC}{\mathbb{C}}
\newcommand{\Da}{\Delta_{\al}}
\newcommand{\Das}{\Delta_{\al}^1}
\newcommand{\eps}{\varepsilon}
\newcommand{\fa}{\mathfrak{a}}
\newcommand{\fb}{\mathfrak{b}}
\newcommand{\fq}{\mathfrak{q}}
\newcommand{\fbg}{\fb_{\G}}
\newcommand{\fhat}{\widehat{f}}
\newcommand{\ghat}{\widehat{g}}
\newcommand{\fhs}{\widehat{f^*}}
\newcommand{\ghs}{\widehat{g^*}}
\newcommand{\Fix}{\operatorname{\mathsf{Fix}}}
\newcommand{\FG}{\Fix_{\G}}
\newcommand{\FGo}{\Fix_{\G}^\circ}
\newcommand{\G}{\Gamma}
\newcommand{\h}{\mathsf{h}}
\newcommand{\hhat}{\widehat{h}}

\newcommand{\lam}{\lambda}
\newcommand{\lizdc}{\ell^\infty(\zd,\CC)}
\newcommand{\lizdr}{\ell^\infty(\zd,\RR)}
\newcommand{\lozdr}{\ell^1(\zd,\RR)}
\newcommand{\lizdz}{\ell^\infty(\zd,\ZZ)}
\newcommand{\m}{\mathsf{m}}
\newcommand{\mf}{\mathfrak{m}_f}
\newcommand{\mfd}{\mathfrak{m}_{f^{(d)}}}
\newcommand{\nfzd}{N(f)\cdot\zd}
\newcommand{\om}{\omega}
\newcommand{\Om}{\Omega}
\newcommand{\OG}{\Om_{\G}}
\newcommand{\OGp}{\Om_{\G}'}
\newcommand{\p}{\mathsf{p}}
\newcommand{\sP}{\mathsf{P}}
\newcommand{\PG}{\sP_{\G}}
\newcommand{\QQ}{\mathbb{Q}}
\newcommand{\rd}{R_d}
\newcommand{\rdma}{\rd/\fa}
\newcommand{\rdmf}{\rd/\<f\>}
\newcommand{\R}{\mathsf{R}}
\newcommand{\RR}{\mathbb R}
\newcommand{\sd}{\SS^{d}}
\newcommand{\SF}{\mathcal{F}}
\newcommand{\sig}{\sigma}
\newcommand{\sigb}{\tilde{\sigma}}
\renewcommand{\SS}{\mathbb{S}}
\newcommand{\TT}{\mathbb{T}}
\newcommand{\tzd}{\TT^{\zd}}
\newcommand{\U}{\mathsf{U}}
\newcommand{\V}{\mathsf{V}}
\newcommand{\VGpC}{V_{\G}'(\CC)}
\newcommand{\VGpR}{V_{\G}'(\RR)}
\newcommand{\VGpZ}{V_{\G}'(\ZZ)}
\newcommand{\VGC}{V_{\G}(\CC)}
\newcommand{\VGR}{V_{\G}(\RR)}
\newcommand{\VGZ}{V_{\G}(\ZZ)}
\newcommand{\vo}{v^{(\bo)}}
\newcommand{\wg}{w^{(g)}}
\newcommand{\xg}{x^{(g)}}
\newcommand{\xig}{\xi_{g}}
\newcommand{\xigb}{\widetilde{\xi}_{g}}
\newcommand{\xrdma}{X_{\rd/\fa}}
\newcommand{\xrdmf}{X_{\rd/\<f\>}}
\newcommand{\ZZ}{\mathbb{Z}}
\newcommand{\zd}{\ZZ^{d}}

\renewcommand{\ge}{\geqslant}
\renewcommand{\le}{\leqslant}
\newcommand{\<}{\langle}
\renewcommand{\>}{\rangle}  
\renewcommand{\emptyset}{\varnothing}
\renewcommand\Im{\operatorname{Im}}
\renewcommand\Re{\operatorname{Re}}

\begin{document}\frenchspacing

\title[Entropy and Growth Rate of Periodic Points]
{Entropy and Growth Rate of \\Periodic Points of Algebraic $\mathbb{Z}^d$-actions}

\author{Douglas Lind}

\address{Douglas Lind: Department of Mathematics, University of
  Washington, Seattle, Washington 98195, USA}
  \email{lind@math.washington.edu}

\author{Klaus Schmidt}

\address{Klaus Schmidt: Mathematics Institute, University of Vienna,
Nordberg\-stra{\ss}e 15, A-1090 Vienna, Austria \newline\indent
\textup{and} \newline\indent Erwin Schr\"odinger Institute for
Mathematical Physics, Boltzmanngasse~9, A-1090 Vienna, Austria}
\email{klaus.schmidt@univie.ac.at}

\author{Evgeny Verbitskiy}

\address{Evgeny Verbitskiy:
Mathematical Institute, University of Leiden, PO Box 9512, 2300 RA Leiden, The Netherlands
\newline\indent \textup{and}
\newline\indent Johann Bernoulli Institute for
   Mathematics and Computer Science, University of Groningen, PO
Box 407, 9700 AK, Groningen, The Netherlands}
\email{e.a.verbitskiy@rug.nl}

\date{\today}





\begin{abstract}
   Expansive algebraic $\zd$-actions corresponding to ideals are
   characterized by the property that the complex variety of the ideal
   is disjoint from the multiplicative unit torus.  For such actions
   it is known that the limit for the growth rate of periodic points
   exists and equals the entropy of the action. We extend this result to
   actions for which the complex variety intersects the multiplicative
   torus in a finite set. The main technical tool is the use of
   homoclinic points which decay rapidly enough to be summable.
\end{abstract}

\maketitle

\section{Introduction}\label{sec:introduction}

An \textit{algebraic $\zd$-action} on a compact abelian group $X$ is a
homomorphism $\al\colon\zd\to\aut(X)$ from $\zd$ to the group of
(continuous) automorphisms of $X$.  We denote the image of $\bn\in\zd$
under $\al$ by $\al^{\bn}$, so that
$\al^{\bm+\bn}=\al^{\bm}\circ\al^{\bn}$ and
$\al^{\mathbf{0}}=\text{Id}_{X}$.

We will consider here cyclic algebraic $\zd$-actions, described as
follows. Let $\rd=\ZZ[u_1^{\pm1},\dots,u_d^{\pm1}]$ denote the ring of
Laurent polynomials with integer coefficients in the variables
$u_1,\dots,u_d$. We write $f\in\rd$ as
$f=\sum_{\bm\in\zd}f_{\bm}\bu^{\bm}$, where $\bu=(u_1,\dots,u_d)$,
$\bm=(m_1,\dots,m_d)\in\zd$, $\bu^{\bm}=u_1^{m_1}\dots u_d^{m_d}$, and
$f_{\bm}\in\ZZ$ with $f_{\bm}=0$ for all but finitely many $\bm$.

Let $\TT=\RR/\ZZ$, and define the shift $\zd$-action $\sig$ on $\tzd$ by
\begin{displaymath}
   (\sig^{\bm}x)_{\bn}=x_{\bm+\bn}
\end{displaymath}
for $\bm\in\zd$ and $x=(x_{\bn})\in\tzd$. For $f=\sum
f_{\bm}\bu^{\bm}\in\rd$ put
\begin{displaymath}
   f(\sig)=\sum_{\bm\in\zd}f_{\bm}\sig^{\bm}\colon \tzd\to\tzd.
\end{displaymath}

We identify $\rd$ with the dual group of $\tzd$ by setting
\begin{displaymath}
   \<f,x\>=e^{2\pi i [f(\sig)x)]_{\mathbf{0}}} =e^{2 \pi i \sum_{\bm}f_{\bm}x_{\bm}}
\end{displaymath}
for $f\in\rd$ and $x\in\tzd$. In this identification the shift
$\sig^{\bm}$ is dual to multiplication by $\bu^{\bm}$ on $\rd$.

A closed subgroup $X\subset\tzd$ is shift-invariant if and only if its
annihilator
\begin{displaymath}
   X^{\perp}=\{h\in\rd:\<h,x\>=1\text{\ for every $x\in X$} \}
\end{displaymath}
is an ideal in $\rd$. In view of this, for every ideal $\fa$ in $\rd$
write
\begin{displaymath}
   \xrdma=\fa^{\perp}=\{x\in\tzd:\<h,x\>=1 \text{\ for every $h\in\fa$}\}
\end{displaymath}
for the closed, shift-invariant subgroup of $\tzd$ annihilated by
$\fa$. Note that the dual group of $\xrdma$ is $\rdma$. Denote by $\ardma$
the restriction of the shift-action $\sig$ on $\tzd$ to $\xrdma$. A
\textit{cyclic algebraic $\zd$-action} is one of this form, corresponding to
the cyclic $\rd$-module $\rdma$.

According to \cite[Eqn.\ (1-1)]{LSW} or \cite[Thm.\ 18.1]{DSAO}, the
topological entropy of $\ardma$, which coincides with its entropy with
respect to Haar measure on $\xrdma$, is given by
\begin{equation}
   \label{eq:entropy}
   \h(\ardma)=
   \begin{cases}
      \infty &\text{if $\fa=\{0\}$}, \\
      \m(f)  &\text{if $\fa=\<f\>=f\cdot\rd$ for some nonzero
      $f\in\rd$,}\\
      0      &\text{if $\fa$ is nonprincipal,}
   \end{cases}
\end{equation}
where
\begin{displaymath}
   \m(f)=\int_0^1\dots\int_0^1\,\log|f(e^{2\pi i t_1},\dots,e^{2\pi i
   t_d})|\,dt_1\dots dt_d
\end{displaymath}
is the \textit{logarithmic Mahler measure} of $f$.

An algebraic $\zd$-action $\al$ on $X$ is \textit{expansive} if there is
a neighborhood $U$ of $0_X$ such that
$\bigcap_{\bm\in\zd}\al^{\bm}(U)=\{0_X\}$. To characterize expansiveness
for cyclic actions $\ardma$, let
\begin{displaymath}
   \V(\fa)=\{(z_1,\dots,z_d)\in (\CC^\times )^d :g(z_1,\dots,z_d)=0
   \text{\ for all $g\in\fa$}\}
\end{displaymath}
denote its complex variety. Let $\SS=\{z\in\CC:|z|=1\}$, so that $\SS^d$
is the unit multiplicative $d$-torus in $(\CC^\times)^d$. Define the
\textit{unitary variety} of $\fa$ as
\begin{displaymath}
   \U(\fa)=\V(\fa)\cap\SS^d =\{(z_1,\dots,z_d)\in\V(\fa):|z_1|=\dots=|z_d|=1\}.
\end{displaymath}
According to \cite[Thm.\ 6.5]{DSAO}, $\ardma$ is expansive if and only
if $\U(\fa)=\emptyset$.

In order to describe periodic points for $\ardmf$, let $\SF$ denote the
collection of finite-index subgroups of $\zd$, and let $\G$ be an
arbitrary element of $\SF$. Define
$\<\G\>=\min\{\|\bm\|:\mathbf{0}\ne\bm\in\G\}$, where
$\|\bm\|=\max\{|m_1|,\dots,|m_d|\}$. A point $x\in X$ has \textit{period
$\G$} if $\al^{\bm}x=x$ for all $\bm\in\G$. Let
\begin{displaymath}
   \FG(\ardma)=\{x\in\xrdma: x \text{\ has period $\G$}\}
\end{displaymath}
be the closed subgroup of $\xrdma$ consisting of all $\G$-periodic
points. In general $\FG(\ardma)$ may be infinite (examples are given in
the next section). We can, however, reduce this to a finite object by
forming the quotient of $\FG(\ardma)$ by its connected component
$\FGo(\ardma)$ of the identity. We therefore define
\begin{displaymath}
   \PG(\ardma)=|\FG(\ardma)/\FGo(\ardma)|,
\end{displaymath}
where $|\cdot|$ denotes cardinality. The growth rate of periodic
components is defined as
\begin{equation}
   \label{eq:components}
   \p^+(\ardma)=\limsup_{\<\G\>\to\infty}\,\frac1{|\zd/\G|}\,\log \PG(\ardma).
\end{equation}

The following relation between entropy and growth rate of $\PG$ was
proved in \cite[Thm.\ 21.1]{DSAO}.

\begin{theorem}
   \label{th:old}
   Let $\fa$ be a nonzero ideal in $\rd$. Then
   $\p^+(\ardma)=\h(\ardma)$. If $\ardma$ is expansive, or equivalently
   if $\U(\fa)=\emptyset$, then the $\limsup$ in \eqref{eq:components}
   is actually a limit, i.e.,
   \begin{equation}
      \label{eq:growth-limit}
      \lim_{\<\G\>\to\infty}\,\frac1{|\zd/\G|}\,\log \PG(\ardma)=\h(\ardma).
   \end{equation}
\end{theorem}

It is not known whether \eqref{eq:growth-limit} holds for all cyclic
actions. Even when $d=1$ the existence of this limit involves some deep
results in number theory (see \cite[Sec.\ 4]{L} for details). The
purpose of this note is to prove the following partial result.

\begin{theorem}
   \label{th:main}
   Let $d\ge2$ and $\fa$ be an ideal in $\rd$ whose unitary variety
   $\U(\fa)$ finite. Then \eqref{eq:growth-limit} holds.
\end{theorem}

The machinery described in \cite[Sec. 21]{DSAO} allows us to reduce the
proof of Theorem \ref{th:main} to the case where the ideal $\fa$ is
prime. If a prime ideal $\fa$ is nonprincipal, then by
\eqref{eq:entropy}
and Theorem \ref{th:old},
$\p^+(\ardma)=\h(\ardma)=0$, which implies \eqref{eq:growth-limit}.
In view of this fact, we can assume from now on that $\fa=\<f\>$ for
some nonzero irreducible Laurent polynomial $f\in\rd$ with
\begin{displaymath}
   |\U(f)|:=|\U(\<f\>)|=|\{\bs\in\SS^d:f(\bs)=0\}|<\infty .
\end{displaymath}
Furthermore we will assume for the remainder of this paper that $d\ge2$,
for reasons which we will explain. In order to simplify notation, we
use $X$ for $\xrdmf$ and $\al$ for $\ardmf$.

\section{Counting periodic components}\label{sec:counting}

In this section we derive an expression for $\PG(\al)=\PG(\ardmf)$ in
terms of $f$.

Let $\G\in\SF$. Following
\cite[(21.13)]{DSAO}, we set $\fbg=\<\bu^{\bm}-1:\bm\in\G\>\subset\rd$,
and
\begin{displaymath}
   \OG=\{\bo=(\om_1,\dots,\om_d)\in\SS^d:\bo^{\bm}=\om_1^{m_1}\cdots\om_d^{m_d}=1
   \text{\ for every $\bm\in\G$}\}.
\end{displaymath}
Observe that $\OG=\U(\fbg)$.
As in \cite[Sec.\ 21]{DSAO}, we note that the dual group of
$\FG(\al)=\FG(\ardmf)$ is
\begin{displaymath}
   \FG(\ardmf)\ \widehat{ }\,=\rd/(\<f\>+\fbg).
\end{displaymath}
Hence $\PG(\al)$ is the cardinality of the $\ZZ$-torsion subgroup of
$\rd/(\<f\>+\fbg)$. The following result shows how to compute this number.

\begin{lemma}
   \label{lem:count}
   For every finite-index subgroup $\G\subset\zd$,
   \begin{displaymath}
      \PG(\al)=\PG(\ardmf)=\prod_{\bo\in\OG\smallsetminus\U(f)}|f(\bo)| .
   \end{displaymath}
\end{lemma}

\begin{proof}
   The group $\rd/\fbg$ is dual to the group $\FG(\sig)$ of
   $\G$-periodic points in $\tzd$. Furthermore, $\FG(\sig)$ is
   isomorphic to the finite-dimensional torus $\TT^{\zd/\G}$. Then
   $\FG(\al)$ is the kernel of the restriction of the homomorphism
   $f(\sig)$ to this torus $\FG(\sig)$.

   To describe this kernel we write $\lizdc$, $\lizdr$, and $\lizdz$ for
   the spaces of bounded complex, real, and integer valued functions on
   $\zd$. Let $\sigb$ be the natural shift-action on each of these
   spaces. Write $\VGC\subset\lizdc$, $\VGR\subset\lizdr$, and
   $\VGZ\subset\lizdz$ for the subspaces of $\G$-invariant elements in
   these spaces.

   Next we diagonalize the restriction of $\sigb$ to $\VGC$. For
   each $\bo=(\om_1,\dots,\om_d)\in\OG$ define $\vo\in\VGC$ by
   \begin{equation}
      \label{eqn:vomega}
      (\vo)_{\bn}:=\bo^{\bn}=\om_1^{n_1}\cdots\om_d^{n_d}
      \text{\quad for $\bn=(n_1,\dots,n_d)\in\zd$}.
   \end{equation}
   Then $\sigb^{\bm}\vo=\bo^{\bm}\vo$ for all $\bm\in\zd$, and so
   $f(\sigb)\vo=f(\bo)\vo$. The set $\{\vo:\bo\in\OG\}$ forms a basis of
   $\VGC$ consisting of eigenvectors of $\sigb$ with distinct
   eigenvalues.

   Let $\OGp=\{\bo\in\OG:f(\bo)\ne0\}=\OG\smallsetminus\U(f)$, and define
   $\VGpC$ to be the $\CC$-linear span of $\{\vo:\bo\in\OGp\}$. Then
   $\VGpC$ is finite-dimensional and $f(\sigb)$-invariant. Observe that
   $f(\sigb)(\VGC)=\VGpC$ and that the restriction of $f(\sigb)$ to
   $\VGpC$ is invertible with
   \begin{displaymath}
      |\det(f(\sigb)|_{\VGpC})|=\prod_{\bo\in\OGp} |f(\bo)|^2.
   \end{displaymath}
   Since $f(\sigb)$ commutes with complex conjugation on $\VGpC$, we can
   restrict it to $\VGpR=\VGpC\cap\lizdr$ and obtain that
   \begin{displaymath}
       |\det(f(\sigb)|_{\VGpR})|=\prod_{\bo\in\OGp} |f(\bo)|.
   \end{displaymath}
   The space $\VGpZ=\VGpR\cap\lizdz$ is a $\sigb$-invariant lattice in
   $\VGpR$, hence $f(\sigb)$-invariant with image
   $f(\sigb)(\VGpZ)\subset\VGpZ$. It follows that
   \begin{equation}
      \label{eq:det}
      |\VGpZ/f(\sigb)(\VGpZ)|=\prod_{\bo\in\OGp}|f(\bo)|.
   \end{equation}
   Finally, we note that $\VGZ\cong\ZZ^{\zd/\G}$ is (isomorphic to) the
   dual group of $\FG(\sig)\subset\tzd$, that
   $\VGZ/f(\sigb)(\VGZ)=\VGZ/f(\sigb)(\VGpZ)$ is dual to $\FG(\al)$, and
   that the torsion subgroup $\VGpZ/f(\sigb)(\VGpZ)$ of
   $\VGZ/f(\sigb)(\VGpZ)$ is therefore dual to $\FG(\al)/\FGo(\al)$. By combining
   this with \eqref{eq:det} we complete the proof.
\end{proof}

\begin{remark}
   Suppose that $\ardmf$ is expansive, so that $\U(f)=\emptyset$. Then
   $f$ does not vanish on $\SS^d$, so $\log|f|$ is continuous there.
   Lemma \ref{lem:count} shows that
   \begin{displaymath}
      \frac1{|\zd/\G|}\,\log \PG(\ardmf)=
      \frac1{|\zd/\G|}\sum_{\bo\in\OG}\log|f(\bo)|
   \end{displaymath}
   is a Riemann sum approximation to $\m(f)$, and so converges to
   $\m(f)=\h(\al)$ as $\<\G\>\to\infty$.

   When $\U(f)\ne\emptyset$ there are two issues to deal with. The
   vanishing of $f$ at some points of $\OG$ creates connected
   components, so we count those. More difficult  are various diophantine
   problems concerning points of $\U(f)$ coming abnormally close to~
   $\OG$. The latter issue is discussed in Section \ref{sec:further}.
\end{remark}

\section{Examples}\label{sec:examples}

We provide here some examples of irreducible polynomials $f$ with finite
$\U(f)$, illustrating a range of algebraic properties of $\U(f)$ and the
resulting influence on the structure of $\FG(\ardmf)$. For clarity we
use variables $u$, $v$, $w$, rather than $u_1$, $u_2$, $u_3$.

\begin{example}
   \label{exam:harmonic}
   Let $d=2$ and $f(u,v)=2-u-v$. Clearly $\U(f)=\{(1,1)\}$. Observe that
   $F=\Fix_{\ZZ^2}(\al_{R_2/\<2-u-v\>})$ is isomorphic to $\TT$, with
   each $t\in\TT$ corresponding to the point in $X_{R_2/\<2-u-v\>}$ all
   of whose coordinates equal~ $t$. For each finite-index subgroup $\G$
   we have that $\OG\cap\U(f)=\{(1,1)\}$, so the analysis of the
   previous section implies that $\Fix_{\G}(\al_{R_2/\<2-u-v\>})$ is a
   finite union of cosets of $F$, and hence always infinite. The exact
   number of these cosets is computed in Lemma \ref{lem:count}.
\end{example}

\begin{example}
   Let $d=2$ and $f(u,v)=1+u+v$. Setting $\om=e^{2\pi i/3}$, it is easy
   to verify that $\U(f)=\{(\om,\om^2),(\om^2,\om)\}$. To describe the
   periodic point behavior of this example, parametrize the
   finite-index subgroups of $\ZZ^2$ as
   \begin{displaymath}
      \G_{a,b,c}=\ZZ\begin{bmatrix}a\\0\end{bmatrix}\oplus\ZZ
      \begin{bmatrix}b\\c\end{bmatrix}, \text{ where $a>0$, $c>0$, and
      $0\le b< a$.}
   \end{displaymath}
   Then
   \begin{displaymath}
      \Om_{\G_{a,b,c}}\cap\U(f)=
      \begin{cases}
         \U(f) &\text{if $a\equiv0$ (mod 3) and $b+2c\equiv 0$
         (mod 3)},\\
         \emptyset &\text{otherwise}.
      \end{cases}
   \end{displaymath}
   Hence $\Fix_{\G_{a,b,c}}(\al_{R_2/\<1+u+v\>})$ is a finite union of
   2-dimensional tori if $a \equiv0$ (mod 3) and $b+2c\equiv0$ (mod 3),
   and is a finite set otherwise. Thus
   $\Fix_{\G_{N\ZZ^2}}(\al_{R_2/\,1+u+v\>})$ is infinite whenever $N$ is
   a multiple of $3$. In this example the coordinates of every point in
   $\U(f)$ are roots of unity.
\end{example}

\begin{example}
   \label{exam:twovar}
   Let $d=2$ and $f(u,v)=2-u^2+v-uv$. We will show that
   $\U(f)=\{(\xi,\eta),(\overline{\xi},\overline{\eta})\}$, where $\xi$
   and $\eta$ are algebraic numbers but not algebraic integers. It
   follows that $\OG\cap\U(f)=\emptyset$ for all $\G\in\SF$,
   and hence that $\FG(\al_{R_2/\<f\>})$ is always finite.

   From $f(u,v)=0$ we obtain that
   $v=v(u)=\frac{2-u^2}{u-1}$. Setting $u=e^{2\pi i\theta}$,
   we must solve $|v(e^{2\pi i \theta})|=1$. Since
   $\overline{u}=u^{-1}=e^{-2\pi i\theta}$, we can write
   $1=|v(u)|^2=v(u)v(\overline{u})=v(u)v(u^{-1})$ as an algebraic
   equation. Clearing fractions yields
   $(2-u^2)(2-u^{-2})=(u-1)(u^{-1}-1)$. Symmetry in $u$ and $u^{-1}$
   means we can write this as an equation in $c=\frac12(u+u^{-1})=\cos
   2\pi\theta$, resulting in $8c^2-2c-7=0$. This equation has roots
   $(1-\sqrt{57})/8\approx -0.818$ and $(1+\sqrt{57})/8\approx
   1.068$. Only the first is a possible value of $\cos 2\pi\theta$, so
   $\Re(\xi)=(1-\sqrt{57})/8$. There are two choices for $\Im(\xi)$,
   namely $\pm(1-\Re{(\xi)}^2)^{1/2}$. Using these yield the
   corresponding values $\eta=v(\xi)$, or explicitly,
   \begin{displaymath}
      \xi=\frac{1-\sqrt{57}}{8}+i \left(\frac{3+\sqrt{57}}{32}\right)^{1/2}
   \end{displaymath}
   and
   \begin{displaymath}
      \eta=\frac{-1}{56+8\sqrt{57}}\left[34+6\sqrt{57}+i\left(11\sqrt{6+2\sqrt{57}}+
            \sqrt{342+114\sqrt{57}}\right)\right]
   \end{displaymath}

   The minimal polynomial for $\xi$ is $2t^4-t^3-3t^2-t+2$ and for
   $\eta$ is $2t^4+13t^3+18t^2+13t+2$, showing that each is an algebraic
   number but not an algebraic integer.
\end{example}

\begin{example}
   Let $d=2$ and $f(u,v)=2-u^3+v-uv-u^2v$. Here $v$ appears linearly,
   and the techniques used in the preceding example still work. In this
   case $v(u)=\frac{2-u^3}{u^2+u-1}$, and the equation $v(u)v(u^{-1})=1$
   is transformed under the change of variables $c=\frac12(u+u^{-1})$ to
   $16c^3-4c^2-12c=0$. The root $c=1$ yields the point
   $(1,1)\in\U(f)$. The root $c=0$ gives $u=\pm i$, with corresponding
   $v(\pm i)=-\frac35\mp i\frac45$. The final root $c=-\frac34$ gives
   $u=-\frac34\pm i \frac{\sqrt{7}}{4}$, with corresponding
   $v=-\frac{528}{704}\pm i\frac{176\sqrt{7}}{704}$. Thus
   \begin{displaymath}
      \U(f)=\{(1,1),\ (i,\xi),\ (-i,\overline{\xi}),\ (\eta,\zeta),\
      (\overline{\eta},\overline{\zeta})\},
   \end{displaymath}
   where
   \begin{displaymath}
      \xi=-\frac35 -i\frac45,\ \  \eta=-\frac34+i\frac{\sqrt{7}}{4},\text{
      and }\ \zeta=-\frac{528}{704}+ i\frac{176\sqrt{7}}{704}
   \end{displaymath}
   are all algebraic numbers but not algebraic integers.

   Note that although $f$ is irreducible, the algebraic properties of
   the coordinates of points in $\U(f)$ vary considerably.
\end{example}

In the previous two examples we exploited the property that one variable
could be expressed as a rational function of the other. In general this
function will be algebraic, and calculations much more difficult. An
alternative approach is to use Gr\"obner bases. Let $u_k=x_k+i y_k$ and
expand $f(x_1+ i y_1,\dots,x_d+ i y_d)$ into real and imaginary parts as
$g(x_1,y_1,\dots,x_d,y_d)+ i h(x_1,y_1,\dots,x_d,y_d)$, where
$g,h\in\ZZ[x_1,y_1,\dots,x_d,y_d]$. Compute a Gr\"obner basis for the
ideal in $\QQ[x_1,y_1,\dots,x_d,y_d]$ generated by $g$, $h$, and the
polynomials $x_k^2+y_k^2-1$ $(1\le k\le d)$, say with term order
$x_1\prec y_1\prec\dots\prec x_d\prec y_d$. If this basis contains a
polynomial in $x_1$ only, we can solve for the real roots and back
substitute to obtain all solutions. Carrying this out on Example
\ref{exam:twovar}, for instance, gives $8x_1^2-2x^{}_1-7$ in the ideal, the
same polynomial (in $c$) as we arrived at there.

Before the next example, we remark that when $d=2$ finding examples is
relatively easy, since generically we expect the 2-dimensional torus to
intersect the (real) 2-dimensional variety in a finite set. This
behavior fails for $d\ge3$, and the matter is more delicate since the
variety must now intersect the torus tangentially in finitely many places.

\begin{example}
   Let $d=3$ and $f(u,v,w)=g(u)-v-w$, where $g(u)=u^4-3u^3+3u+3$.

   We claim that the minimum value of $|g|$ on $\SS$ is 2, and that this
   minimum is attained at exactly two algebraic integers $\eta$ and
   $\overline{\eta}$ in $\SS$. It turns out here that
   $g(\eta)=2\overline{\eta}$. Hence
   \begin{displaymath}
      \U(f)=\{(\eta,\overline{\eta},\overline{\eta}),
              (\overline{\eta},\eta,\eta)\}.
   \end{displaymath}
   It follows that all periodic point groups are finite.

   To verify our claim, use the rational function parametrization
   $s\colon\RR\to\SS\smallsetminus\{-i\}$ given by
   \begin{displaymath}
      s(t)=\frac{2t}{1+t^2}+i\,\frac{1-t^2}{1+t^2}.
   \end{displaymath}
   (Omitting $-i$ from the range is harmless since $-i$ is far from the
   location of the minimum.) Then
   \begin{displaymath}
      \phi(t)=|g(s(t))|^2 = g(s(t))g(\overline{s(t)})>0.
   \end{displaymath}
   Expanding this product and taking the derivative shows, after a
   lengthy calculation, that
   \begin{displaymath}
      \phi'(t)=-\frac{96}{(1+t^2)^5} (t^8-7t^7-10t^6+25t^5-25t^3+10t^2+7t-1).
   \end{displaymath}
   Evaluating $\phi$ at the real roots of $\phi'(t)=0$ shows that the minimum
   value of $\phi$ is attained at the two real roots of the irreducible
   quartic factor $t^4+t^3-2t^2+t+1$ of the numerator of $\phi'(t)$,
   explicitly at
   \begin{displaymath}
      \xi=\frac{-1-\sqrt{17}}{4}+\frac12 \sqrt{\frac{1+\sqrt{17}}{2}}
   \end{displaymath}
   and its real conjugate. We put $\eta=s(\xi)\in\SS$. An exact calculation
   shows that $g(\eta)=2\overline{\eta}$, verifying our claim.

   One variation on this theme is to use $f(u,v,w)=g(u)-v^r-w^s$, which
   has a more complicated, but still finite, unitary variety.
\end{example}

\section{Algebraic points on varieties}\label{sec:algebraic-points}

In every example from the preceding section, the coordinates of the
points in $\U(f)$ are algebraic numbers. Using an argument kindly shown
to us by Marius van der Put, we will prove that this is always true. The
algebraicity of the coordinates is crucial to our proof of
Theorem~\ref{th:main}.

We begin with a result in real algebraic geometry.

\begin{proposition}
   \label{prop:real-alg}
   Let $\fq$ be an ideal in $\QQ[t_1,\dots,t_d]$ and define
   \begin{displaymath}
      \R(\fq):=\{(r_1,\dots,r_d)\in\RR^d:g(r_1,\dots,r_d)=0
      \text{\ for all $g\in\fq$}\}.
   \end{displaymath}
   Suppose that $(a_1,\dots,a_d)$ is an isolated point in
   $\R(\fq)$. Then each $a_j$ is an algebraic number.
\end{proposition}

\begin{proof}
   Each $\ba=(a_1,\dots,a_d)\in\R(\fq)$ gives a ring homomorphism
   \begin{displaymath}
      \phi_{\ba}\colon \QQ[t_1,\dots,t_d]/\fq\to K := \QQ(a_1,\dots,
      a_d)\subset\RR
   \end{displaymath}
   with $\phi_{\ba}(t_j)=a_j$, and every homomorphism $\QQ[t_1,\dots,
   t_d]/\fq\to\RR$ comes from a point in $\R(\fq)$ this way.

   Suppose that $\ba=(a_1,\dots,a_d)\in\R(\fq)$, and that
   $K=\QQ(a_1,\dots,a_d)$ is not algebraic over $\QQ$. Then there are
   $k\ge1$ algebraically independent elements $b_1,\dots,b_k\in K$ and
   an element $c\in K$ algebraic over $\QQ(b_1,\dots,b_k)$ such that
   $K=\QQ(b_1,\dots,b_k)(c)$. Write the minimal polynomial of $c$ over
   $\QQ(b_1,\dots,b_k)$ as
   \begin{displaymath}
      P(b_1,\dots,b_k,T) :=T^n +p_{n-1}(b_1,\dots,b_k)T^{n-1}+\dots+p_0(b_1,\dots,b_k)\
   \end{displaymath}
   where the $p_j(T_1,\dots,T_k)\in\QQ(T_1,\dots,T_k)$ are rational
   functions. Now $P(b_1,\dots,b_k,T)$ is irreducible over
   $\QQ(b_1,\dots,b_k)\subset\RR$, so that $c$ is a simple root. Hence
   there are $c_1<c$ and $c_2>c$ such that $P(b_1,\dots,b_k,c_1)$ and
   $P(b_1,\dots,b_k,c_2)$ are nonzero and have opposite sign. Therefore
   if we perturb slightly each $b_j$ to $b_j'$, the new polynomial
   $P(b_1',\dots,b_k',T)\in\RR[T]$ has a root $c'$ very close to ~$c$. If
   we further assume that $\{b_1',\dots,b_k'\}$ is also algebraically
   independent, then there is a field isomorphism
   \begin{displaymath}
      \psi\colon \QQ(b_1,\dots,b_k,c)\to\QQ(b_1',\dots,b_k',c').
   \end{displaymath}

   Now each $a_j$ is in $K=Q(b_1,\dots,b_k)(c)$ and can thus be written in
   the form
   \begin{displaymath}
      a_j=\sum_{m=0}^n q_{mj}(b_1,\dots,b_k)c^m, \text{\ where
      $q_{mj}(T_1,\dots,T_k)\in\QQ(T_1,\dots,T_k)$.}
   \end{displaymath}
   Hence if the perturbations are sufficiently small, we see that
   \begin{displaymath}
      a_j':=\sum_{m=0}^n q_{mj}(b_1',\dots,b_k')(c')^m
   \end{displaymath}
   is very close to $a_j$ for $1\le j\le d$. Let
   $\ba'=(a_1',\dots,a_d')$. Then
   \begin{displaymath}
      \phi_{\ba'}=\psi\circ\phi_{\ba}\colon \QQ[t_1,\dots,t_d]/\fq\to\RR
   \end{displaymath}
   is a homomorphism, and so $\ba'\in\R(\fq)$. This proves that if $\ba$
   has at least one non-algebraic coordinate, then $\ba$ cannot be
   isolated in $\R(\fq)$.
\end{proof}

\begin{proposition}
   \label{prop:alg}
   Let $f\in\rd$ and suppose that $\U(f)$ is finite. Then the
   coordinates of every point in $\U(f)$ are algebraic numbers.
\end{proposition}

\begin{proof}
   We again use the rational function parametrization
   $s\colon\RR\to\SS\smallsetminus\{-i\}$ given by
   \begin{displaymath}
       s(t)=\frac{2t}{1+t^2}+i\,\frac{1-t^2}{1+t^2}.
    \end{displaymath}
    Define $\bs\colon\RR^d\to\SS^d$ by
    $\bs(t_1,\dots,t_d)=(s(t_1),\dots,s(t_d))$. We may assume that
    $\U(f)\subset\bs(\RR^d)$. For if this fails, we can easily adjust
    the parametrization to omit a point on $\SS$ with rational
    coordinates that does not appear as a coordinate of any point in the
    finite set $\U(f)$.

    Consider the equation $f(\bs(t_1,\dots,t_d))=0$. Expanding and
    multiplying through by $\prod_{k=1}^d(1+t_k)^{n_k}\ne0$ for suitable
    $n_k$, this takes the form
    \begin{displaymath}
       g_1(t_1,\dots,t_d)+i\,g_2(t_1,\dots,t_d)=0,
    \end{displaymath}
    where each $g_j\in\ZZ[t_1,\dots,t_d]$. Let
    $\fq=\<g_1,g_2\>\subset\QQ[t_1,\dots,t_d]$. By assumption, $\R(\fq)$
    is finite, so all of its points are isolated. By the preceding
    proposition, these points have algebraic coordinates. Each point in
    $\U(f)$ is the image under $\bs$ of a point in $\R(\fq)$, and hence
    also has coordinates that are algebraic numbers.
\end{proof}

\section{Homoclinic points}\label{sec:homoclinic}

In this section we will construct periodic points by using homoclinic
points which decay rapidly enough.

Let $\beta$ be an algebraic $\zd$-action on a compact abelian group
$Y$. An element $y\in Y$ is \textit{homoclinic for $\beta$} if
$\lim_{|\bn|\to\infty}\beta^{\bn}y=0_Y$. The set of all homoclinic
points for $\beta$ is a subgroup of $Y$, denoted by $\Delta_{\beta}(Y)$.

According to \cite{LS}, the following hold if $\beta$ is assumed to be
expansive:
\begin{enumerate}
  \item $\Delta_{\beta}(Y)$ is at most countable;
  \item $\Delta_{\beta}(Y)\ne\{0_Y\}$ if and only if $\beta$ has
   positive entropy with respect to Haar measure $\lam_Y$ on $Y$;
  \item  $\Delta_{\beta}(Y)$ is dense in $Y$ if and only if $\beta$ has
   completely positive entropy with respect to $\lam_Y$; and
  \item For every $y\in\Delta_{\beta}(Y)$, $\beta^{\bn}y\to0_Y$
   exponentially fast.
\end{enumerate}

If $\beta$ is not expansive, then there is no guarantee that
$\Delta_{\beta}(Y)\ne\{0\}$, even if $\beta$ has completely positive
entropy. For example, let $A\in \text{\textit{GL}}_n(\ZZ)$ have
irreducible characteristic polynomial, and also have some but not all of
its roots on $\SS$. Then by \cite[Example\ 3.4]{LS}, the $\ZZ$-action
generated by $A$ on $\TT^n$ has completely positive entropy (indeed is
Bernoulli), and yet has trivial homoclinic group.

Furthermore, if $\beta$ is not expansive then homoclinic points may
decay very slowly, in contrast to the exponential decay in the expansive
case. Let $f(u,v)=2-u-v$ and consider the $\ZZ^2$-action
$\al_{R_2/\<f\>}$ on $X_{R_2/\<f\>}$ that we discussed in Example~
\ref{exam:harmonic}. By \cite[Example\ 7.2]{LS}, $1/\overline{f}$ is
integrable on $\SS^2$, and its Fourier transform $w^{\Delta}$ is given by
\begin{equation}
   \label{eqn:hcpt}
   w_{(-m,-n)}^{\Delta}=
   \begin{cases}
      \displaystyle\frac1{2^{m+n+1}} \binom{m+n}{n}
                                      &\text{if $m\ge0$ and $n\ge0$},\\
      0                               &\text{otherwise.}
   \end{cases}
\end{equation}
Let $x^{\Delta}$ denote the coordinate-wise reduction (mod 1) of
$w^{\Delta}$. Then $x^{\Delta}$ is a homoclinic point for
$\al_{R_2/\<f\>}$, and in fact every homoclinic point is a finite
integral combination of translates of $x^{\Delta}$. Note that
\begin{displaymath}
   w_{(-n,-n)}^{\Delta}= \displaystyle\frac1{2^{2n+1}} \binom{2n}{n}
   \approx \frac1{2\sqrt{\pi n}}
\end{displaymath}
decays slowly, and also that $\sum_{m,n}|w_{(m,n)}^{\Delta}|=\infty$.

When $\U(f)$ is finite but nonempty, the action $\al=\ardmf$ is not
expansive on $X=\xrdmf$. We will restrict our attention to those
homoclinic points which decay rapidly enough to be summable. Hence
define
\begin{displaymath}
   \Das(X):=\biggl\{x\in\Da(X):\sum_{\bn\in\zd}|x_{\bn}|<\infty\biggr\},
\end{displaymath}
where for $t\in\TT$ we let $|t|$ denote the distance from $t$ to $0$
in $\TT$.

In order to analyze the homoclinic group, we first linearize the action
$\al$. Consider the surjective map $\rho\colon \lizdr\to\tzd$ given by
$\rho(w)_{\bn}= w_{\bn} \,\text{(mod 1)}$. If
$f=\sum_{\bn}f_{\bn}\bu^{\bn}$ and $\sigb$ is the shift-action on
$\lizdr$, then $f(\sigb)=\sum_{\bn}f_{\bn}\sigb^{\bn}\colon
\lizdr\to\lizdr$. We define
\begin{displaymath}
   \begin{aligned}
      W_f:=\rho^{-1}(X)&= \{w\in\lizdr:\rho(w)\in X\} \\
         &=\{w\in\lizdr:f(\sigb)(w)\in\lizdz \},
   \end{aligned}
\end{displaymath}
and view $W_f$ as the linearization of $X$.

For $w\in\lizdr$, we define its \textit{adjoint} $w^*$ by
$w^*_{\bn}=w^{}_{-\bn}$. Each $a\in\lozdr$ acts as a linear operator on
$\lizdr$ via convolution, defined by
\begin{displaymath}
   (a*w)_{\bm}=\sum_{\bn\in\zd} a_{\bn}w_{\bm-\bn} \text{\quad for all $w\in\lizdr$}.
\end{displaymath}

For $a\in\lozdr$ we define its Fourier transform
$\widehat{a}\colon\SS^d\to\CC$ by
$\widehat{a}(\bs)=\sum_{\bn}a_{\bn}\bs^{\bn}$, where as usual
$\bs^{\bn}=s_1^{n_1}\cdots s_d^{n_d}$. In the opposite direction, if
$\phi\colon\SS^d\to\CC$ is integrable with respect to Haar measure
$\lam$ on $\SS^d$, then we write $\widetilde{\phi}\in\lizdc$ for its
Fourier transform, where
\begin{displaymath}
   \widetilde{\phi}_{\bn}=\int_{\SS^d}\phi(\bs)\bs^{-\bn}\,d\lam(\bs).
\end{displaymath}

If $g=\sum_{\bn}g_{\bn}\bu^{\bn}\in\rd$, we can consider $g$ as the
element $(g_{\bn})\in\lozdr$. With this convention, the action of
$g(\sigb)$ on $\lizdr$ coincides with convolution by $g^*$, i.e.,
$g(\sigb)(w)=g^* *w$. Furthermore, $\widehat{g}$ is just the restriction
of the polynomial function $g$ to $\SS^d$, and $\widehat{g^*}$ is the
restriction of the complex conjugate $\overline{g}$.

Since the Fourier transform $\fhat$ of $f$ has only finitely many zeros
on $\SS^d$ by assumption, it follows that $1/\fhat\colon\SS^d\to\CC$ is
analytic with finitely many poles. We seek multipliers that will make
the Fourier transform summable, and so define
\begin{displaymath}
   \mf:=\bigl\{g\in\rd:\ghat/\fhat\colon\SS^d\to\CC\text{ has absolutely
   convergent Fourier series}\bigr\},
\end{displaymath}
which is clearly an ideal in $\rd$. For every $g\in\mf$ we write
\begin{equation}
   \label{eqn:multiplier}
   \wg=(\ghs/\fhs)\,\widetilde{ }\,\in\lozdr
\end{equation}
for the summable Fourier transform of $\ghs/\fhs=\overline{\ghat/\fhat}$.

\begin{proposition}
   \label{prop:mfproper}
   Let $f\in\rd$ with finite $\U(f)$. Then $\<f\>\subsetneq\mf$.
\end{proposition}

Before beginning the proof, we remark that if $\U(f)=\emptyset$, then
$1/\fhat$ is smooth, and so $\mf=\rd$. However, if $\U(f)$ is nonempty,
then $1/\fhat$ is not bounded, and so $\mf$ is a proper ideal. The
strict containment $\<f\>\subsetneq\mf$ fails for $d=1$, and this is the
main reason we require $d\ge2$.

\begin{proof}
   First note that $f$ cannot be expressed as a polynomial in fewer than
   $d$ variables since $\U(f)$ is finite. Hence no polynomial of fewer
   variables can be contained in~ $\<f\>$.

   Define the isomorphism $e\colon\TT^d\to\SS^d$ by
   $e(t_1,\dots,t_d)=(e^{2\pi i t_1},\dots,e^{2\pi i t_d})$. As before,
   for $t\in\TT$ let $|t|$ denote the distance from $t$ to 0.  For
   $\bt,\bt'\in\TT^d$ put $\|\bt-\bt'\|=\max\{|t_j-t_j'|:1\le j\le
   d\}$. Define the metric $\delta$ on $\SS^d$ by
   $\delta(\bs,\bs')=\|e^{-1}(\bs)-e^{-1}(\bs')\|$.

   Let $\ba=(a_1,\dots,a_d)\in\U(f)$. Since $\fhat\circ e$ is
   analytic on $\RR^d$, and in a neighborhood of $e^{-1}(\ba)$ vanishes only there,
   it follows that there are constants $c>0$, $k\ge1$, and $\eps>0$ such
   that
   \begin{displaymath}
      |\fhat(\bs)|\ge c\,\delta(\bs,\ba)^k \text{\ whenever
      $\delta(\bs,\ba)<\eps$. }
   \end{displaymath}

   We start by considering the first coordinate $a_1$ of $\ba$. By Proposition
   \ref{prop:alg}, $a_1$ is an algebraic number. Hence there is a
   nonzero polynomial $h_1\in\ZZ[u_1]$ with $h_1(a_1)=0$. It follows that
   $|\hhat(s)|\le c_1\delta_1(s,a_1)$ for $s\in\SS$ near $a_1$, where
   $c_1>0$ is a suitable constant and $\delta_1$ is the metric on $\SS$
   analogous to $\delta$. Define $h\in\rd$ by
   $h(u_1,\dots,u_d)=h_1(u_1)$. Then for $\bs$ near $\ba$ we have that
   $|\hhat(\bs)|\le c_1\,\delta_1(s_1,a_1)\le
   c_1\,\delta(\bs,\ba)$. Hence near $\ba$ we have the estimate
   \begin{displaymath}
      \Biggl| \frac{\hhat^n(\bs)}{\fhat(\bs)}\Biggr|\le
      \Bigl(\frac{c_1^n}{c}\Bigr)\,\delta(\bs,\ba)^{n-k}.
   \end{displaymath}
   By taking $n$ sufficiently large we can guarantee that $\hhat^n/\fhat$
   is as differentiable as we please, in particular that it is $d$ times
   continuously differentiable.

   Repeating this procedure for every point in
   $\U(f)$, and letting $g[u_1]\in\ZZ[u_1]$ be the product of the
   corresponding $h_1^n(u_1)$'s, we obtain that $\ghat/\fhat$ is $d$ times
   continuously differentiable on $\SS^d$. Hence the Fourier series of
   $\ghat/\fhat$ is absolutely convergent (see \cite{summable1} or
   \cite{summable2} for much sharper results). Thus $g\in\mf$, and since
   it is a polynomial in one variable it cannot be in $\<f\>$ by our
   earlier remark.
\end{proof}

\begin{proposition}
   \label{prop:wg}
   Suppose that $f\in\rd$ has finite $\U(f)$, and let $\al=\ardmf$ be
   the algebraic $\zd$-action on $X=\xrdmf$. For every $g\in\mf$ let
   $\xg$ = $\rho(\wg)$, where $\wg$ is defined in
   \eqref{eqn:multiplier}. Then $\Das(X)=\{\xg:g\in\mf\}$. Furthermore,
   $\xg=0$ if and only if $g\in\<f\>$, so that the map
   $g+\<f\>\mapsto\xg$ is a group isomorphism of $\mf/\<f\>$ with $\Das(X)$.
\end{proposition}

\begin{proof}
   Let $z\in\Das(X)$. Choose $w\in\lozdr$ with $\rho(w)=z$. Then
   \begin{displaymath}
      f(\sigb)(w)\in\lizdz\cap\lozdr,
   \end{displaymath}
   and so is an element, say $g^*$, of $\rd$. Taking Fourier transforms
   of $f^* *w=f(\sigb)(w)=g^*$ shows that $\fhs\cdot\widehat{w}=\ghs$. Hence
   $\widehat{w}=\ghs/\fhs$ is well-defined off a finite set, and has
   absolutely convergent Fourier series $w$. Thus $w=\wg$, and so
   $z=\xg$.

   Conversely, suppose that $g\in\mf$. Then $\wg=(\ghs/\,\fhs)\,\widetilde{
   }\in\lozdr$, and as above we obtain that
   $f(\sigb)(\wg)=g^*\in\lizdz$. Hence $\wg\in W_f$, and so
   $\xg=\rho(\wg)\in\Das(x)$.

   Finally, if $g\in\mf$ and $\xg=\rho(\wg)=0$, then
   $h=\wg\in\rd$. Taking Fourier transforms gives $\hhat=\ghs/\fhs$, so
   that $g^*=h\cdot f^*$ and $g=f\cdot h^*\in\<f\>$. The converse is obvious.
\end{proof}


Sometimes it is useful to determine $\mf$ explicitly.  For
example, this is the case in \cite{SV}, where the Laplacian
$f^{(d)}=2d-\sum_{j=1}^d(u_j^{}+u_j^{-1})\in R_d$, $d\ge 2$, was
studied. There it is shown that
\begin{equation}
   \label{eq:result_ideal}
   \mfd=\<f^{(d)}\>+\mathscr{I}^3_d,
\end{equation}
where $\mathscr{I}_d=\bigl\{h\in R_d:h(1,\dots,1)=0\bigr\}=
\<u_1 -1,\dots,u_d -1\>$.

We demonstrate how to obtain such results using again the example
$f(u,v)=2-u-v\in R_2$ discussed in Example \ref{exam:harmonic} and at
the start of Section \ref{sec:homoclinic}.  Firstly, since $f$ has only
one zero on $\SS^2$, namely $\mathbf 1=(1,1)$, the Fourier
transform $\fhat$ has one zero on $\mathbb T^2\cong [-1/2,1/2)^2$ at
$\mathbf 0=(0,0)$. The Taylor series
expansion of $\fhat$ at $\mathbf0$ is
\begin{displaymath}
   \fhat(\theta,\phi) =
   - 2\pi i(\theta+\phi)+2\pi^2(\theta^2+\phi^2)+\mathcal O(|\theta|^3+
   |\phi|^3).
\end{displaymath}
According to the proof of Proposition \ref{prop:mfproper},
$g_m(u):=(u-1)^m\in\mf$ for all sufficiently large~ $m$. What is the minimal such $m$?
If $g_m(u)\in\mf$, then $\ghat_m/\fhat$ must be at least
continuous at $\mathbf 0$.
Inspecting the Taylor series expansion of $\ghat_m$ at $\mathbf 0$
for small $m$ we find that
\begin{align*}
   \ghat_0(\theta,\phi) &=1,\\
   \ghat_1(\theta,\phi) &=-2\pi i \theta+2\pi^2\theta^2+\mathcal
   O(\theta^3),\\
   \ghat_2(\theta,\phi) &=-4\pi^2\theta^2+\mathcal O(\theta^3).
\end{align*}
It is evident that $\ghat_m/\fhat$ is not continuous at
$\mathbf0$ for $m=0,1,2$. It turns out that $g_3(u)\in \mf$. We can
establish this fact either by showing that $\ghat_3/\fhat$ is
sufficiently smooth at $\mathbf 0$, or alternatively by showing that
$g^*_3*w^\Delta \in\lozdr$, where $w^\Delta$ is the homoclinic point
given by \eqref{eqn:hcpt}.

For every $(m,n)\in\ZZ^2$ we have that
\begin{displaymath}
   (g^*_3*w^\Delta)_{(m,n)}=
   -w_{(m,n)}+3w_{(m+1,n)}-3w_{(m+2,n)}+w_{(m+3,n)}.
\end{displaymath}
Assuming that $m\ge 3$, $n\ge 0$, and using expression \eqref{eqn:hcpt}
for elements of $w^\Delta$, one has after some manipulation that
\begin{displaymath}
   (g_3^* * w^{\Delta})_{(-m,-n)}=\frac1{2^{m+n+1}}\binom{m+n}{m}
   \frac{(m-n)^3-3(m^2-n^2)+2(m-n)}{(m+n-2)(m+n-1)(m+n)}.
\end{displaymath}

Let $N=m+n>3$. Then $m=N-n$, and so
\begin{displaymath}
   |(g_3^* * w^{\Delta})_{(-m,-n)}|\le \frac1{2^{N+1}}\binom{N}{n}
   \frac{|N-2n|^3+(3N+2)|N-2n|}{(N-2)(N-1)N}.
\end{displaymath}
Suppose $X_1,X_2,\ldots$ are independent random variables with an
identical distribution $\mathbb P(X_i=\pm1)=\frac 12$. Then, by a well
known probabilistic result, the so-called Khintchine inequality
\cite{PS}, for any $p>0$ there exists a constant $c_p$ such that
$$
\mathbb E\biggl|\sum_{i=1}^N X_i\biggr|^p =\frac 1{2^N}\sum_{n=0}^{N}\binom{N}{n}|N-2n|^p\le c_pN^{\frac p2}\quad \text{for all }N.
$$
Thus for some $C>0$ and all sufficiently large $N$
\begin{displaymath}
   \sum_{\begin{subarray}{1} m\ge 3, n\ge0 \\ m+n=N \end{subarray}}
   |(g_3^* * w^{\Delta})_{(-m,-n)}|  \le \frac{C}{N^{3/2}}.
\end{displaymath}
This, together with the observation that the boundary terms
$w^{\Delta}_{(-m,-n)}$ with $m\le3$ are exponentially small in $N=m+n$,
proves that $g_3^* * w^{\Delta}\in\lozdr$.

Similarly one shows that other third powers $(1-u)^2(1-v)$,
$(1-u)(1-v)^2$, $(1-v)^3$ belong to $\mf$ as well. Moreover,
$u-1\equiv-(v-1)\ \text{mod}\ \<f\>$ and $\mf\supset\<f\>$. Therefore
from $(u-1)^2\notin\mf$ we conclude that $(u-1)(v-1)\notin\mf$ and
$(v-1)^2\notin\mf$.
Thus we have exactly
identified the multiplier ideal of $f(u,v)=2-u-v$ to be $ \mf=\<f\>
+{\mathscr I}_2^3$.

\section{Symbolic covers}\label{sec:symbolic}

For every nonzero summable homoclinic point $z\in\Das(X)$ we construct
here a shift-equivariant group homomorphism from $\lizdz$ to
$X$. Indeed this map is surjective when restricted to a ball of finite
radius in $\lizdz$, and so provides a symbolic cover of $X$.

According to Proposition $\ref{prop:wg}$, every homoclinic point
$z\in\Das(X)$ has the form $z=\rho(\wg)$ for some $g\in\mf$, where
$\wg\in\lozdr$. We define group homomorphisms
$\xigb\colon\lizdz\to\lizdr$ and $\xig\colon\lizdz\to\tzd$ by
\begin{displaymath}
   \xigb(v)={\wg}^*(\sigb)(v)=\wg*v \text{\ \ and\ \ } \xig(w)=\rho(\xigb(w)).
\end{displaymath}
These maps are well-defined since $\wg\in\lozdr$, and commute with the
appropriate $\zd$-actions.

\begin{proposition}
   \label{prop:cover}
   For every $g\in\mf$,
   \begin{equation}
      \label{eqn:onto}
      \xig(\lizdz)=
      \begin{cases}
         \{0\} &\text{if $g\in\<f\>$}.\\
         X     &\text{if $g\in\mf\smallsetminus\<f\>$}.
      \end{cases}
   \end{equation}
\end{proposition}

We first establish two lemmas.

\begin{lemma}
   For every $v\in\lizdr$ and $g\in\mf$,
   \begin{equation}
      \label{eqn:conv}
      (f(\sigb)\circ\xigb)(v)=f^* *\wg *v=g^* *v=g(\sigb)(v).
   \end{equation}
\end{lemma}

\begin{proof}
   The proof of Proposition \ref{prop:wg} shows, after taking
   Fourier transforms, that $\eqref{eqn:conv}$ holds whenever $g\in\mf$
   and $v\in\lozdr$.

   For $K\ge1$ put $V_K=\{v\in\lizdr:\|v\|_{\infty}\le K\}$. Then $V_K$
   is shift-invariant and compact in the topology of pointwise
   convergence, and the set $V_K^1=\V_K\cap\lozdr$ is dense in
   $V_K$. For $v\in V_K^1$ clearly $\xigb(v)=\wg *v$ and
   $(f(\sigb)\circ\xigb)(v)=g^* *v$. Since $\xigb$ and $f(\sigb)$ are
   continuous on $V_K$, these equations continue to hold for all $v\in
   V_K$. Letting $K\to\infty$ shows that \eqref{eqn:conv} holds for all
   $v\in\lizdr$ and $g\in\mf$.

   For the last assertion, recall that for every $v\in \rd$,
   \begin{displaymath}
      f(\sigb)(\xigb(v))=f^* *\wg *v=g^* *v\in \rd\subset\lizdz.
   \end{displaymath}
   Hence $\xig(v)=\rho(\xigb(v))\in X$ for every $v\in\rd$. The
   continuity argument above then shows that $\xig(v)\in X$ for every
   $v\in\lizdz$.
\end{proof}

\begin{lemma}
   \label{lem:onto}
   Let $g\in\mf\smallsetminus\<f\>$, and put
   $K=\sum_{\bn\in\zd}|f_{\bn}|$. Then $\xig(V_K)=X$, and so
   $\xig(\lizdz)=X$. Furthermore, the restriction of $\xig$ to $V_K$,
   or to any other bounded, closed, shift-invariant subspace of
   $\lizdz$, is continuous in the product topology.
\end{lemma}

\begin{proof}
   Let $x\in X$. Choose $w\in W_f$ with $\rho(w)=x$ and $0\le w_{\bn}<1$
   for all $\bn\in\zd$. If $v=f(\sigb)(w)$, then $v\in\lizdz$ and $-K\le
   v_{\bn}\le K $ for every $\bn\in\zd$, so that $v\in V_K$.

   Since $\xigb$ commutes with $f(\sigb)$, we see that
   $\xig(v)=\rho(\xigb(v))=g(\al)(x)$. This shows that
   $g(\al)(X)\subset\xig(V_K)\subset X$.

   We claim that $g(\al)(X)=X$. For $h+\<f\>\in\rdmf$ annihilates $g(\al)X$
   iff $gh+\<f\>$ annihilates $X$ iff $gh\in\<f\>$ iff $h\in\<f\>$, since
   $f$ is irreducible and $g\notin\<f\>$. This shows that $g(\al)(X)$ and $X$ have
   the same annihilator, and so $g(\al)(X)=X$.

   Continuity of $\xig$ follows as in the previous lemma.
\end{proof}

\begin{proof}[Proof of Proposition \ref{prop:cover}]
   If $g=h\cdot f\in\<f\>$ for some $h\in\rd$, then $\wg=h^*\in\rd$, and
   hence $\xigb(v)=h*v\in\lizdz$ for every $v\in\lizdz$, showing that
   $\xig(\lizdz)=\{0\}$. The case $g\in\mf\smallsetminus\<f\>$ is handled by
   Lemma \ref{lem:onto}.
\end{proof}

\section{Proof of Theorem \ref{th:main}}\label{sec:proof}

We use the fact that entropy equals the growth rate of separated sets,
and that by using homoclinic points we can approximate elements in such
sets with periodic points.

\begin{lemma}
   \label{lem:folner}
   Let $\{\G_n\}_{n\ge1}$ be a sequence of finite-index subgroups of
   $\zd$ with $\<\G_n\>\to\infty$ as $n\to\infty$. Then there exists a
   sequence $\{Q_n\}_{n\ge1}$ of subsets of $\zd$ such that
   \begin{enumerate}
     \item Each $Q_n$ is a fundamental domain for $\G_n$, i.e. the
      collection $\{Q_n+\bm:\bm\in G_n\}$ is disjoint and has union
      $\zd$; and
     \item $\{Q_n\}_{n\ge1}$ is a F{\o}lner sequence in $\zd$.
   \end{enumerate}
\end{lemma}

\begin{proof}
   This is an easily proved special case of \cite[Cor.\
   5.6]{DS}.
\end{proof}

\begin{definition}
   Let $Q\subset\zd$ and $\eps>0$. We say that $E\subset X$ is
   \textit{$(Q,\eps)$-spanning} in ~$X$ if, for every $x\in X$ there is a
   $y\in E$ such that $|x_{\bn}-y_{\bn}|<\eps$ for every $\bn\in
   Q$. Dually, $F\subset X$ is \textit{$(Q,\eps)$-separated} in $X$ if,
   for every distinct pair $x,y$ of points in ~$F$, there is an $\bn\in
   Q$ with $|x_{\bn}-y_{\bn}|\ge\eps$.
\end{definition}

\begin{lemma}
   \label{lem:span}
   For every $\eps>0$ there exists a finite set $A_\eps$ with the
   following property: if $\G$ is a finite-index subgroup of $\zd$ and
   $Q$ is a fundamental domain for $\G$, then $\FG(\al)$ is
   $\bigl(\bigcap_{\bm\in A_{\eps}}(Q-\bm),\eps\bigr)$-spanning in $X$.
\end{lemma}

\begin{proof}
   Fix $g\in\mf\smallsetminus\<f\>$, and define $\wg\in\lozdr$ as in
   \eqref{eqn:multiplier}. Let $\eps>0$, and put
   $K=\sum_{\bn\in\zd}|f_{\bn}|$. Choose a finite subset $A_\eps$ of
   $\zd$ so that $\sum_{\bn\in\zd\smallsetminus A_\eps}|\wg_{\bn}|<\eps/K$.

   Since $\xigb(V_K)=X$ by Proposition \ref{prop:cover}, for every $x\in
   X$ there is a $v\in V_K$ with $\xig(v)=x$. Define $v'\in\VGZ$ by
   requiring that $v_{\bn}'=v_{\bn}$ for every $\bn\in Q$, and extending
   $v'$ by $\G$-periodicity. Our choice of
   $A_\eps$ implies that $|\xigb(v)_{\bn}-\xigb(v')_{\bn}|<\eps$ for every
   $\bn\in\bigcap_{\bm\in A_{\eps}}(Q-\bm)$. Let $x'=\rho(v')$. Then
   $x\in \FG(\al)$ and $|x^{}_{\bn}-x'_{\bn}|<\eps$ for every
   $\bn\in\bigcap_{\bm\in A_{\eps}}(Q-\bm)$.
\end{proof}

We write
\begin{displaymath}
   \Om(f)=\{\bo=(\om_1,\dots,\om_d)\in\U(f):\text{ each $\om_j$ is a
   root of unity}\}
\end{displaymath}
for the set of torsion points in $\U(f)$. If $\Om(f)\ne\emptyset$, set
\begin{displaymath}
   \G(f)=\{\bn\in\zd:\bo^{\bn}=1:\text{ for every $\bo\in\Om(f)$}\}.
\end{displaymath}
Then $\G(f)\in\SF$, and we can find $N(f)>0$ with $\G(f)\subset N(f)\cdot\zd$.

\begin{lemma}
   \label{lem:subtorus}
   Let $\G$ be a finite-index subgroup, and put
   $\Om_{\G}(f)=\Om(f)\cap\OG$. Then $\FG(\al)$ is finite if and only if
   $\Om_{\G}(f)=\emptyset$. If $\Om_{\G}(f)\ne\emptyset$, then
   $\FGo(\al)\cong\TT^{|\Om_{\G}(f)|}$ and $\FGo(\al)\subset\Fix_{\nfzd}(\al)$.
\end{lemma}

\begin{proof}
   We denote by $W_{\G}(\CC)\subset\lizdc$ the linear span of
   $\{\vo:\om\in\Om_{\G}(f)\}$, where $\vo$ is defined in
   \eqref{eqn:vomega}.  Write $\VGR=W_{\G}(\CC)\cap\lizdr\subset W_f$
   for the real part of $W_{\G}(\CC)$. The dimension of $\VGR$ equals
   $|\Om_{\G}(f)|$, and
   $\FGo(\al)=\rho(\VGR)\cong\TT^{|\Om_{\G}(f)|}$. Since $\VGR\subset
   V_{\nfzd}(\RR)$, applying $\rho$ shows that $\FGo(\al)\subset
   \Fix_{\nfzd}(\al)$.
\end{proof}

\begin{lemma}
   \label{lem:card}
   For every $\eps>0$ there is an $M(\eps)>0$ with the following
   property:
   for each $\G\in\SF$, every
   $(\zd,\eps)$-separated set in $\FGo(\al)$ has cardinality $<M(\eps)$.
\end{lemma}

\begin{proof}
   By Lemma \ref{lem:subtorus}, for every $\G\in\FG$ we have that
   $\FGo(\al)$ is a subtorus of the fixed finite-dimensional torus
   $\Fix_{\nfzd}(\sig)$. If $Q=\{0,\dots,N(f)-1\}^d$, there is an
   $M(\eps)>0$ such that every $(Q,\eps)$-separated set in
   $\Fix_{\nfzd}(\sig)$ has cardinality $<M(\eps)$. By periodicity,
   every $(\zd,\eps)$-separated set in $\Fix_{\nfzd}(\sig)$ (and hence
   in $\FGo(\al)$) has cardinality $<M(\eps)$.
\end{proof}

\begin{lemma}
   \label{lem:cosets}
   Let $Q\subset\zd$ and $\G$ be a finite-index subgroup of
   $\zd$. Suppose that $\eps>0$ and that $F\subset\FG(\al)$ is a
   $(Q,\eps)$-separated subset with cardinality $L$. Then $F$ intersects
   at least $L/M(\eps)$ distinct cosets of $\FGo(\al)$ in $\FG(\al)$,
   where $M(\eps)$ is given in Lemma \ref{lem:card}.
\end{lemma}

\begin{proof}
   This is an immediate consequence of Lemma \ref{lem:card}.
\end{proof}

\begin{proof}[Proof of Theorem \ref{th:main}]
   For a finite subset $Q\in\zd$, let $r_Q(\eps)$ denote the largest
   cardinality of a $(Q,\eps)$-separated set in $X$. According to
   \cite[Prop.\ 2.1]{D}, for every F{\o}lner sequence $\{L_n\}_{n\ge1}$
   in $\zd$, we have that
   \begin{equation}
        \label{eqn:liminf}
      \lim_{\eps\to 0}\,\liminf_{n\to\infty}\frac1{|L_n|}\log r_{L_n}(\eps)=\h(\al).
   \end{equation}

   Let $\{\G_n\}_{n\ge1}$ be a sequence in $\SF$ with $\<\G_n\>\to\infty$
   as $n\to\infty$. By Lemma \ref{lem:folner}, there exists a F{\o}lner
   sequence $\{Q_n\}_{n\ge1}$ of fundamental domains for the $\G_n$.

   Fix $\eps>0$ and use Lemma \ref{lem:span} to find a finite set
   $A_{\eps/3}\subset\zd$ such that $\Fix_{\G_n}(\al)$ is
   $(Q_n',\eps/3)$-spanning in $X$ for every $n\ge1$, where
   $Q_n'=\bigcap_{\bm\in A_{\eps/3}}(Q_n-\bm)$. Note that
   $\{Q_n'\}_{n\ge1}$ is again a F{\o}lner sequence in $\zd$ with
   $|Q_n'|/|Q_n|\to1$ as $n\to\infty$. We may assume $Q_n'\ne\emptyset$
   for all $n\ge1$.

   For all $n\ge1$ choose a maximal $(Q'_n,\eps)$-separated set
   $F_n\subset X$ with cardinality $r_{Q_n'}(\eps)$. We fix $n$ for the
   moment and choose for every $y\in F_n$ an element
   $z(y)\in\Fix_{\G_n}(\al)$ with $|y_{\bn}-z(y)_{\bn}|<\eps/3$ for all
   $\bn\in Q_n'$. The points $z(y)$ must be distinct for different $y$,
   so $F_n'=\{z(y):y\in F_n\}$ also has cardinality $|F_n|$. Lemma
   \ref{lem:cosets} shows that there is an $M(\eps)>0$ (which depends on
   $\eps$ but not on $n$) such that $F_n'$ intersects at least
   $|F_n'|/M(\eps/3)$ distinct cosets of $\Fix^{\circ}_{\G_n}(\al)$ in
   $\Fix_{\G_n}(\al)$. Hence
   \begin{displaymath}
      \sP_{\G_n}(\al)=|\Fix_{\G_n}(\al)/\Fix_{\G_n}^{\circ}(\al)|
      \ge\frac{|F'_n|}{M(\eps/3)}=\frac1{M(\eps/3)} r_{Q'_n}(\eps)
   \end{displaymath}
   for every $n\ge1$. It follows that
   \begin{displaymath}
      \liminf_{n\to\infty}\frac1{|\zd/\G_n|}\log \sP_{\G_n}(\al)
      \ge\liminf_{n\to\infty}\frac1{|Q_n|} \log r_{Q_n'}(\eps)
      =\liminf_{n\to\infty}\frac1{|Q_n'|} \log r_{Q_n'}(\eps).
   \end{displaymath}
   Letting $\eps\to 0$, invoking \eqref{eqn:liminf}, and combining with
   Theorem \ref{th:old} completes the proof.
\end{proof}

\section{Specification}\label{sec:specification}

Specification is a strong orbit tracing property that has many
uses. Ruelle \cite{R} investigated the extension of this notion to
topological $\zd$-actions. In \cite{LS} it was shown that expansive
algebraic $\zd$-actions with completely positive entropy satisfy several
flavors of specification. The proof made crucial use of the existence of
summable homoclinic points. By Proposition \ref{prop:wg}, this tool
remains available for the (nonexpansive) actions $\ardmf$ when $\U(f)$ is
finite. In this section we extend previous results to such actions.

\begin{definition}
   Let $\beta$ be a $\zd$-action by homeomorphisms of a compact metric
   space $(X,\rho)$.

   (1) The action $\beta$ has \textit{strong specification} if there
   exists, for every $\eps>0$, a number $p(\eps)>0$ with the following
   property: for every finite collection $\{Q_1,\dots,Q_t\}$ of finite
   subsets of $\zd$ with
   \begin{equation}
      \label{eqn:sep}
      \text{dist}(Q_j,Q_k) := \min_{\bm\in Q_j,\ \bn\in
      Q_k}\|\bm-\bn\|\ge p(\eps)\ \ (1\le j<k\le t),
   \end{equation}
   every collection $\{x^{(1)},\dots,x^{(t)}\}\subset X$, and every
   $\G\in\SF$ with
   \begin{displaymath}
      \text{dist}(Q_j+\mathbf{k},Q_k)\ge p(\eps) \ \ (1\le j<k\le
      t,\ \mathbf{k}\in\G\smallsetminus\{\boldsymbol{0}\}),
   \end{displaymath}
   there is a $y\in\FG(\beta)$ with
   \begin{equation}
      \label{eqn:spec}
      \rho(\beta^{\bn}y,\beta^{\bn}x^{(j)})<\eps \text{\ for all $\bn\in
      Q_j$, $(1\le j \le t)$}.
   \end{equation}

   (2) The action $\beta$ has \textit{homoclinic specification} if, for
   every $\eps>0$, there is a $p(\eps)>0$ such that for every finite
   collection $\{Q_1,\dots,Q_t\}$ of finite subsets of $\zd$ satisfying
   \eqref{eqn:sep} and every $\{x^{(1)},\dots,x^{(t)}\}\subset X$, there
   is a $y\in\Delta_{\beta}(X)$ satisfying \eqref{eqn:spec}.
\end{definition}

\begin{theorem}
   Let $d\ge2$ and $f\in\rd$ have finite $\U(f)$. Then $\ardmf$ has both
   strong specification and homoclinic specification.
\end{theorem}

\begin{proof}
   Let $\al=\ardmf$ and $X=\xrdmf$. Using the notation of Proposition
   \ref{prop:wg}, choose $g\in\mf\smallsetminus\<f\>$, with
   corresponding $x^{(g)}\in\Das$. By the proof of Lemma \ref{lem:onto},
   we can find $y^{(j)}\in X$ with $g(\al)(y^{(j)})=x^{(j)}$ for $1\le
   j\le t$. The proof of \cite[Thm.\ 5.2]{LS}, applied to the $y^{(j)}$
   and replacing the fundamental homoclinic point with $x^{(g)}$,
   now yields the required $y$.
\end{proof}

\begin{remark}
   For $d=1$ and $\U(f)\ne\emptyset$, both strong specification and
   homoclinic specification always fail, although a weaker form still
   holds \cite{L2}. This again illustrates the difference between $d=1$
   and $d\ge2$.
\end{remark}

\section{Further remarks}\label{sec:further}

An alternative approach to proving Theorem \ref{th:main} uses Gelfond's
deep results on algebraic numbers (see \cite[p. 28]{G}). Let
$\bxi=(\xi_1,\dots,\xi_d)$ have algebraic coordinates, and recall that
these are called \textit{multiplicatively independent} if the only
$\bn\in\zd$ for which $\bxi^{\bn}:=\xi_1^{n_1}\cdots\xi_d^{n_d}=1$ is
$\bn=\boldsymbol{0}$.

\begin{theorem}[Gelfond, {\cite[Thm. III]{G}}]
   Suppose that $\bxi=(\xi_1\dots,\xi_d)$ has algebraic coordinates that
   are multiplicatively independent. Then for every $\eps>0$ there are
   only finitely many $\bn\in\zd$ for which
   $|\bxi^{\bn}-1|<e^{-\eps\|\bn\|}$, where $\|\bn\|=\max\{|n_1|,\dots,|n_d|\}$.
\end{theorem}

Let $f\in\rd$ have finite $\U(f)$. Define $\log_0 t$ for $t\ge0$ to be
$\log t$ if $t>0$ and $0$ if $t=0$. According to Lemma \ref{lem:count},
for each $\G\in\SF$,
\begin{equation}
   \label{eq:riemann-sum}
   \frac1{|\OG|} \log \PG(\ardmf)= \frac1{|\OG|} \sum_{\bo\in\OG}\log_0|f(\bo)|.
\end{equation}
Now $\log|f|$ has only finitely many logarithmic singularities, and by
Proposition \ref{prop:alg} these all have algebraic coordinates. We can
therefore use Gelfond's result to control the few potentially large
negative values of $\log_0|f|$ for $\bo\in\OG$ near one of these
singularities, to show that the Riemann sums in \eqref{eq:riemann-sum}
will converge to the limit $\m(f)=\h(\ardmf)$.

To make a similar argument work when $\U(f)$ is infinite, we would need
an estimate of the form $\text{dist}(\U(f),\bo)\ge e^{-\eps\cdot
o(\bo)}$, where $\bo=(\om_1,\dots,\om_d)\notin\U(f)$ has coordinates
which are roots of unity, and $o(\bo)$ denotes the order of $\bo$ in
$\sd$. Such estimates, however, do not appear to be available.

We remark that if we replace averages of $\log|f|$ over the finite
subgroups $\OG$ by averages over a sequence $\{K_n\}$ of compact
connected subgroups in $\sd$ that become uniformly distributed, then a
result of Lawton \cite{La} shows that for every nonzero $f\in\rd$ we do
have convergence,
\begin{displaymath}
   \int_{K_n}\log|f|\,d\lambda_{K_n} \to
   \int_{\sd}\log|f|\,d\lambda_{\sd} \text{ as $n\to\infty$},
\end{displaymath}
and so the diophantine issues disappear.

\section*{Acknowledgment}
We are  grateful to Marius van der Put for the proof of Proposition \ref{prop:real-alg}.
K.S. and E.V. gratefully acknowledge the support and hospitality
of Max-Planck-Institut f\"ur Mathematik in Bonn, Germany, where part of this
work was done. D.L. thanks the Mathematics Department at Yale
University for their support during a visit.

\end{document}